\documentclass[12pt,leqno]{article}
\usepackage{amsmath,amsfonts,amsthm,amscd,amssymb,mathrsfs,esint}
\usepackage{xcolor,bbm,easybmat,bm}
\usepackage[colorlinks,citecolor=red,pagebackref,hypertexnames=false]{hyperref}
\usepackage{enumerate}
\usepackage{graphicx}
\usepackage{esint} 
\usepackage{dsfont}

\numberwithin{equation}{section}

\newtheorem{lem}{Lemma}[section]
\newtheorem{pro}[lem]{Proposition}

\newtheorem{thm}[lem]{Theorem}
\newtheorem{cor}[lem]{Corollary}

\newcommand{\ms}{\medskip}

\newcommand{\R}{\mathbb{R}}
\newcommand{\bN}{\mathbb{N}}
\newcommand{\bZ}{\mathbb{Z}}
\newcommand{\bC}{\mathbb{C}}

\renewcommand{\H}{\mathscr{H}}

\renewcommand{\d}{\partial}
\newcommand{\dist}{\,\mathrm{dist}\,}

\newcommand{\sm}{\setminus}

\newcommand{\diam}{\mathrm{diam}}

\newcommand{\wt}{\widetilde}

\newcommand{\cY}{{\mathcal Y}}
\newcommand{\cE}{{\mathcal E}}

\newcommand{\ol}{\overline}

\newcommand{\wh}{\widehat}

\newcommand{\mres}{\mathbin{\vrule height 1.6ex depth 0pt width
0.13ex\vrule height 0.13ex depth 0pt width 1.3ex}}

\begin{document}

\title{Cantor sets with absolutely continuous harmonic measure}

\author{
Guy David\footnote{
G.~David was partially supported by the European Community H2020 grant GHAIA 777822,
and the Simons Foundation grant 601941, GD.
} 
\, Cole Jeznach\footnote{
C.~Jeznach was partially supported by the Simons Collaborations in MPS grant 563916, SM and Max Engelstein's NSF grant 2000288.
}
\, and Antoine Julia\footnote{
A.~Julia was supported by the Simons Foundation grant 601941, GD.
}
}

\newcommand{\Addresses}{{
  \bigskip
  \footnotesize

 Guy David, \textsc{Universit\'e Paris-Saclay, CNRS, 
 Laboratoire de math\'ematiques d'Orsay, 91405 Orsay, France} 
 \par\nopagebreak
  \textit{E-mail address}: \texttt{guy.david@universite-paris-saclay.fr} \\
  
 Cole Jeznach, \textsc{School of Mathematics, University of Minnesota, Minneapolis,
MN, 55455, USA} 
 \par\nopagebreak
  \textit{E-mail address}: \texttt{jezna001@umn.edu}\\
  
 Antoine Julia, \textsc{Universit\'e Paris-Saclay, CNRS, 
 Laboratoire de math\'ematiques d'Orsay, 91405 Orsay, France} 
 \par\nopagebreak
  \textit{E-mail address}: \texttt{antoine.julia@universite-paris-saclay.fr} \\
  
 }}

\maketitle

\abstract{
  We construct Ahlfors regular Cantor sets $K$ of small dimension in the plane, such that the Hausdorff measure on $K$ is equivalent to the harmonic measure associated to its complement. In particular the Green's function in $\R^2 \sm K$ satisfies $G^p(x)\simeq \dist(x,K)^\delta$ whenever $\dist(x,K)\le 1$ and $p$ is far from $K$. 

}
\tableofcontents

\section{Introduction}\label{S1}
Given a connected domain $\Omega$ in the plane, whose complement has positive capacity and a point $x\in \Omega$, a fundamental question of harmonic analysis is whether the harmonic measure $\omega^x$ of $\Omega$ with pole at $x$ is comparable to the Hausdorff measure on the boundary of $\Omega$, $\sigma:= \mathscr{H}^\delta \mres \partial \Omega$, where $\delta>0$ is the dimension of $\partial \Omega$. It can mean several things for measures to be comparable: do they have the same dimension? Are they mutually absolutely continuous? If so, is the density $d\omega^x/d \sigma$ bounded from above and away from zero?  

\medskip

Jones and Wolff proved in \cite{JonesWolff} 
that the dimension of $\omega^x$ is at most one, i.e.~there exists a set of Hausdorff dimension one which has 
full $\omega^x$ measure. This answers the question in case $\partial \Omega$ has a 
dimension larger than one (for instance if the set is bounded by a Koch snowflake). 
In higher ambient dimensions, the question is more subtle, but Azzam recently proved in \cite{Azzam2020} that if $E\subset \R^n$ has dimension $s>n-1$ and is $s$-Ahlfors regular, then the harmonic measure on the complement of $E$ has dimension smaller than $s$. 

\medskip

When the boundary of $\Omega \subset \R^n$ is of co-dimension $1$,
the mutual absolute continuity of $\omega^x$ with respect to 
$\sigma = \mathscr{H}^{n-1} \mres \partial \Omega$ is known to be mostly
a question of non-tangential accessibility and rectifiability. There is a large collection of sufficient conditions 
that imply the mutual absolute continuity of $\omega^x$ and $\sigma$; see for instance \cite{AHMMT2020Inventiones}. 
More recently, converse results were also found, and in particular Azzam 
et al.~proved in \cite{AHMMMTV2016GAFA} that if $\omega^x$
is absolutely continuous with respect to $\sigma$ on a piece of $\d\Omega$
with finite $\sigma$-measure, then this piece is $(n-1)$-rectifiable. We 
refer to \cite{AHMMMTV2016GAFA} for a complete review of this question.

\medskip

In this note, we mostly 
focus on boundary sets $K \subset \R^2$ of dimensions smaller than $1$. 
When $K = \d\Omega$ is a self-similar Cantor set of dimension $\delta\in (0,1)$, it 
was shown by Carleson \cite{Carleson1985}, 
Makarov and Volberg \cite{MakarovVolberg1986}, 
and Batakis \cite{Batakis1996} that 
$\omega^x$ is supported by a Borel set of dimension smaller than $\delta$; 
see also \cite{BatakisZdunik2015} for a generalization, and 
\cite{Volberg1992, Volberg1993, Zdunik1997} for similar questions related to complex dynamics.
However, restricting to a set on which the harmonic measure concentrates, Batakis constructed a Cantor set 
of dimension $\delta$ for which $\omega^x$ is also $\delta$-dimensional. 
Furthermore, in an unpublished letter Bishop constructed a non self-similar Cantor set of small dimension 
on which $\sigma$ and $\omega$ are boundedly equivalent. 
But these particular sets are very much not Ahlfors regular, and it 
was recently conjectured by Volberg \cite[Conjecture~1.9]{Volberg2022onephase} 
that this is impossible for Ahlfors regular sets of dimension strictly smaller than $1$. 
Even more recently, X.~Tolsa \cite{Tolsa2023ADreg} proved this conjecture in the special case of a set $K$ contained in the line and with Hausdorff dimension at least $1/2-\varepsilon$. 
Tolsa's result and the case of self similar sets show that Volberg's conjecture is natural. 
However, we prove that it does not hold:

\ms

\begin{thm}\label{thm:main}
  There exists $\delta_M\in (0,1\,]$ such that for $\delta\in (0,\delta_M)$, there exists a compact Ahlfors regular set $K$ of dimension $\delta$, such that for $x\in \Omega := \R^2 \sm K$, 
   the harmonic measure $\omega^x$ on $\Omega$, with pole at $x$, is 
   equivalent to $\sigma = \mathscr{H}^\delta \mres K$, i.e.
  \begin{equation} \label{eq:measures-eq}
C_x^{-1} \sigma(A) \leq \omega^x(A) \leq C_x \sigma(A)
\ \text{ for } A\subset K,
  \end{equation}
    for some $C_x \geq 1$.
  In addition, the Green function $G^x$ of $\Omega$ with pole at $x$ satisfies:
  \begin{equation}\label{eq:green-dist}
 C_x^{-1} \dist(y,K)^\delta \leq  G^x(y) \leq C_x \dist(y,K)^\delta,
  \end{equation}
 for $y$ close enough to $K$. 
 Furthermore, if $\delta$ is smaller than a certain $\delta_1 \in (0,1/2)$  
 the set $K$ can be contained in a line.
\end{thm}

\ms

Recall that a compact set $K \subset \R^n$ is called Ahlfors regular
of dimension $\delta >0$ if it contains at least $2$ points and
there is a positive measure $\mu$ on $K$ and a constant $C_0 \geq 1$ such that 
\begin{equation}\label{1a1}
C_0^{-1} \rho^\delta \leq \mu(B(x,\rho)) \leq C_0 \rho^\delta
\ \text{ for $x\in K$ and } 0 < \rho \leq \diam(K).
\end{equation}
It is well known that then $\mu$ is equivalent to the restriction to $K$
of the Hausdorff measure $\H^\delta$, in the sense that
\begin{equation}\label{1a2}
C_1^{-1} \H^\delta \mres  K \leq \mu \leq C_1 \H^\delta \mres K
\end{equation}
for some constant $C_1$ that depends on $n$, $C_0$ and $\delta$. 
\ms

To be precise, Volberg conjectured that sets for which \eqref{eq:green-dist} holds are necessarily of dimension $\delta=1$, provided their complement satisfies an accessibility condition ($\Omega=\R^2\backslash K$ should be a one-sided NTA domain, see for instance \cite{DFM-Memoirs}). 
Both these conditions hold for our sets. 
 Volberg also conjectured that it is impossible for the ratio of two Green's functions to be analytic on such sets. We do not know if, nor believe that our set can provide a counterexample to that conjecture.

\ms

The two estimates \eqref{eq:measures-eq} and \eqref{eq:green-dist}
are related, because when $K \subset \R^2$ is Ahlfors regular of dimension 
$\delta < 1$, its complement is automatically connected, with non-tangential access.
See for instance Section~2 of \cite{DFM-Memoirs} 
for definitions and a proof (even for Ahlfors regular sets of dimension $< n-1$ in $\R^n$).
Because of this, we can estimate the size of the Green function at 
corkscrew points in terms of the harmonic measure of a corresponding surface ball. 
 This is quite standard for Cantor sets (see Lemma \ref{lem:MV}), but we also refer to the more general result 
 \cite{DFM-Mixed}, Lemma 15.28, 
which applies to our case and variants of the construction. 
This observation was used in \cite{DM-Compactness}, 
with the hope that characterizations of uniform rectifiability for sets $K$ may
be easier in terms of approximation of the Green function by distance functions.

\ms

It should be observed that our examples are close to being self-similar, as
they are bi-Lipschitz equivalent to regular fractal sets, and that for self-similar sets  
$\omega$ is singular to $\sigma$. 
This contrasts with the situation of co-dimension $1$, 
in which  all the absolute continuity criteria that we know are bi-Lipschitz invariant.

\medskip 

In the statement above, we can make sure that the Ahlfors constant $C_0$ stays bounded 
if we restrict to $\delta \in [\delta_{min},\delta_M]$ for some $\delta_{min} > 0$. 
We can also normalize $K$ so that $\diam(K) = 1$, and then we can make sure that the constants $C_x$ do not depend on $\delta$ and $x$, provided $\dist(x,K) \geq 1$, and $\dist(y, K) \leq 1$, say.

 \ms

We will prove the second part of the theorem with the dimension bound $\delta_1 = 0.249$, which is the best we could do on the line. 
This can probably be improved, but by X.~Tolsa's result there necessarily holds $\delta_1 < 1/2$, if $K$ is contained in a line.
We will not give a detailed proof for sets not contained in the line, but only show how to adapt the strategy. Overall, the best lower bound we were able to obtain for $\delta_M$ is a little more than $0.4$. 

\ms

We stated our result in terms of harmonic measure and Green function with a pole
$x$ at finite distance (and then it is known that the precise choice of $x$ hardly matters),
but the analogous result with a pole at infinity follows, and in fact this is what we will prove
first, because the Green function with pole at infinity is easier to construct.

\ms

For Theorem \ref{thm:main}, we insist on using the usual Laplace operator. It is easier to find other elliptic operators $L$ adapted to $K$ whose elliptic measure
satisfies \eqref{eq:measures-eq} and \eqref{eq:green-dist}. 
Indeed, it was proved in \cite{DavidMayboroda-good}, that even 
for the standard self-similar four-corner Cantor set of dimension~$1$, 
we can find an adapted elliptic operator 
$L= -\rm{div} \, a \, \nabla$,  with scalar coefficients $a$, 
for which the analogue of \eqref{eq:measures-eq} and \eqref{eq:green-dist} holds. 
The present work is also related to the study of elliptic operators adapted to boundaries of 
higher co-dimensions which has been developed in recent years (say, after  \cite{DFM-Memoirs}).

\ms 

We will explain in Section \ref{Sn} how to prove the following analogue of
Theorem \ref{thm:main} in $\R^n$, $n > 2$.

\begin{thm}\label{thm:main2}
 For each $n \geq 3$, there exists $\delta_n\in (0,1\,]$ such that for $\delta\in (0,\delta_n)$, 
 we can find a compact Ahlfors regular set $K$ of dimension $n-2+\delta$, 
 such that for $x\in \Omega := \R^2 \sm K$, 
   the harmonic measure $\omega^x$ on $\Omega$, with pole at $x$, is 
   equivalent to $\sigma = \mathscr{H}^{n-2+\delta} \mres K$, i.e.
  \begin{equation} \label{eq:measures-eq2}
C_x^{-1} \sigma(A) \leq \omega^x(A) \leq C_x \sigma(A)
\ \text{ for } A\subset K,
  \end{equation}
    for some $C_x \geq 1$.
  In addition, the Green function $G^x$ of $\Omega$ with pole at $x$ satisfies:
  \begin{equation}\label{eq:green-dist2}
 C_x^{-1} \dist(y,K)^\delta \leq  G^x(y) \leq C_x \dist(y,K)^\delta,
  \end{equation}
 for $y$ close enough to $E$. 
\end{thm}

In Theorem \ref{thm:main}, we could consider sets $K$ of any (small enough)
dimension $\delta \in (0,1)$, because any Ahlfors regular set of dimension $\delta > 0$ has a 
positive logarithmic capacity, and hence the harmonic measure is well defined. 
In $\R^n$, $n \geq 3$, this is no longer true,
which is why we restrict to sets of dimensions larger than $n-2$.

Also, the sets $K$ provided by Theorem \ref{thm:main2} 
are of the form $K_0 \times S$, where $K_0$ is an Ahlfors regular set of dimension $\delta$ in a line $L$,
and $S$ is the unit sphere of dimension $n-2$ in the orthogonal hyperplane $L^\perp$. Thus $K$ is
no longer a Cantor set, but it remains that it is not rectifiable (its dimension is not an integer).

If we really wanted to use the same sort of construction to find a Cantor set 
$K \subset \R^n$ with the desired properties, it seems that we would need to find
configurations $A$ of $N$ points in (or near) the sphere, with $N$ large enough
(so that $K$ has a dimension larger than $m-2$) and symmetry properties such that
the potential created at $a\in A$ by the other points of $A$ is the same for all $a \in A$.
This looks like the Thomson problem of optimizing the electric energy of $N$ charges on the sphere, 
maybe with a bit more flexibility, but at this point we do not know how to proceed.

\ms

\subsection*{Construction strategy}

We now say a few words about the algorithm for proving Theorem~\ref{thm:main}. 
A usual way to obtain an Ahlfors regular set $K^0$ of dimension
$\delta \in (0,1)$ in the line is to use a self-similar construction. We start from the 
unit interval $I_0 = [0,1]$, then replace $I_0$ with two intervals $I$ of length $r < 1/2$
contained in $I_0$, at each end of $I_0$, then replace each of these two intervals by
two intervals of length $r^2$, and so on. At the $n$-th generation, 
we get $2^n$ intervals of length $r^n$, and the intersection of all the generations yields a self-similar Cantor set of dimension $\delta$, where $\delta$ is such that $\delta = -\ln(2)/\ln(r)$.

\ms 

If we allow sets $K^0 \subset \R^2$ that are not contained in the line, 
we can start with a square and keep a smaller square in each corner, or even start with a disk and keep 
$N$ equally spaced smaller disks tangent to the boundary. This leaves more flexibility, and we shall discuss in 
Section \ref{sec:Npieces} how to start from such a construction to get Cantor sets $K \subset \R^2$ with slightly 
larger dimensions than the ones that we construct in the line.

\ms

The set $K^0 \subset \R$ described above is self-similar and the results of \cite{BatakisZdunik2015} 
imply that the harmonic measure $\omega^x$ on $\R^2 \sm K^0$ 
lives on a subset of $K^0$ of dimension $< \delta$. One way to interpret this is to say that
at finite scales, there will always be intervals of the decomposition that receive 
a tiny bit more harmonic measure $\omega$ than their fair share.
These would typically be intervals that sit at the 'exterior' of $K^0$.
By self-similarity, this unbalance will grow geometrically and we expect $\omega$
to be concentrated on the parts of $K^0$ which are close to the 'exterior' at many steps in the construction of $K$.

\ms

However, there is some room in this construction: one can translate the intervals a little bit along the real 
line without losing the Ahlfors regularity. The key remark is that if two neighboring intervals are far from the rest of the set, 
moving them apart increases the probability that they will be hit by Brownian motion. (From the analytic point of view, 
the logarithmic capacity of a union of two sets increases with the distance between these two sets.) 
Thus, when splitting an interval that was getting less than its share of harmonic measure we can increase 
the harmonic measure of its descendants by spacing them a bit more than what is done 
in the rest of the construction, and this will turn 
out to be enough to compensate the unbalance of the previous step. 

\ms

 The problem is that computing exactly the harmonic measure is difficult and any error should be immediately compensated, 
 if one does not want it to grow out of control as in the self-similar case.
 Fortunately,  our construction allows us to compute the Green function. 
 To do this, we use classical potential theory, which tells us that a probability measure (a charge distribution) on $K$ 
 whose logarithmic (electric) potential is constant on $K$ is equal to the harmonic measure with pole at infinity. 
 This measure is called the \emph{equilibrium distribution} of $K$. 
 Moreover, the corresponding potential\footnote{The potential is the convolution $e \ast \mu$ with the fundamental solution of the Laplacian in dimension $2$,  $e(x) = \ln(|x|)$.} 
 is the Green function at infinity, up to an additive constant. 
 We just need to find this measure $\mu$ and show that it is equivalent to the Hausdorff measure. 
 We will construct $K$ so that the equilibrium distribution $\mu$ is precisely the measure given by the pushforward onto $K$ of the natural probability measure on the abstract Cantor set $\{0,1\}^{\bN}$. Together with the fact that $\mu$ satisfies \eqref{1a1} and thus \eqref{1a2}, this will allow us to conclude.
 
 \ms

 Now let us define approximating Green's functions. Instead of considering finite unions of intervals, 
 we approximate our Cantor set by finite families $K_n$ of  $2^n$ points (the centers of the intervals above)
 and estimate the electric potential created by a probability measure $\mu_n$ evenly distributed on these points. 
 For each point $x\in K_n$, we will denote by $g_n(x)$ the potential created by all the other points in $K_n$. 
 If $g_n$ is close to constant on $K_n$ we will construct $K_{n+1}$ so that $g_{n+1}$ oscillates even less: this is possible because given $x\in K_{n+1}$ the main contribution to $g_{n+1}(x)$ is that of the sibling $\wt x$ of $x$ (the only other point in $K_{n+1}$ with the same parents), which can be tuned by spacing $x$ and $\wt x$ more or less apart. This will suffice 
to compensate exactly the bias inherited from the variation of $g_n$. Of course another bias is created at step $n+1$, but it is not very large because the potential created by $x$ and $\wt x$ is not very different from that created by their parent, and depends little on the spacing. It can thus be controlled at the next step and altogether the oscillation of $g_n$ will go to $0$ 
as $n$ tends to infinity. 
The measures $\mu_n$ converge to a natural probability measure $\mu$ and the corresponding potentials will 
converge to the potential created by $\mu$, which is constant on $K$. 
Thus $\mu$ is the harmonic measure and its potential is (a constant plus) the Green 
function for $\R^2\sm K$.

 We shall discuss later in Subsection \ref{sec:Npieces} 
 some slightly different ways to construct Cantor sets $K \subset \R^2$ that are not contained in a line,
 and that yield Ahlfors regular sets of slightly larger dimensions that satisfy the properties of Theorem \ref{thm:main}.
 We shall also see in Section \ref{Sn} how to construct examples $K \subset \R^n$, $n \geq 3$. In the 
 mean time we explain with more detail the geometric construction of $K_n$ and $K \subset \R$
 (in Section \ref{sec:construct}), then use estimates on the Green function to prove 
 Theorem \ref{thm:main} (in Section \ref{sec:Green}).
\ms

\section{Cantor set construction}\label{sec:construct}

\subsection{Construction on the line}\label{sec:construct-line}

 Our Cantor set $K \subset \R$ will be an embedding of the ``universal Cantor set'' 
 $\cE :=\{-1,1\}^{\bN \backslash \{0\}}$. 
 Let $r \in (0, 1/2)$ and $a \in [1,3]$ be given, to be chosen later. 
 We define a map $f : \cE \to \R$ by
 \begin{equation}\label{2a1}
f(\varepsilon):= \sum_{k=1}^\infty \dfrac{a_{k-1}(\varepsilon)}{2} \, r^{k-1} \varepsilon_k,
\end{equation}
where for each infinite word $\varepsilon= (\varepsilon_1,\varepsilon_2,\dots) \in \cE$ and each $k \geq 0$,
we chose a coefficient $a_k(\varepsilon)=a_k(\varepsilon_1,\dots,\varepsilon_{k}) \in [\,1, a\,]$.
For the first term we take $a_0(\varepsilon) = 1$.
It will be important that $a_k(\varepsilon)$ (to be chosen later) can depend on the $k$ first co-ordinates of
$\varepsilon$, but not on the others. Then set
\begin{equation}\label{eq:defK}
    K:= \Big \{  f(\varepsilon)\, , \, \varepsilon\in \cE \Big \}.
\end{equation}
 We will also use the approximations $\cE_n:= \{-1,1\}^{\{1,\dots,n\}}$
 of $\cE$ (which can be seen as sets of words of length $n$) and the approximating sets $K_0 = \{0\}$ and,
 for $n \geq 1$, 
 \[
   K_n:= \Big \{ f_n(\varepsilon):= 
   \sum_{k=1}^{n} \dfrac{a_{k-1}(\varepsilon)}{2} \, r^{k-1} \varepsilon_k, \, \varepsilon\in \cE_n \Big \}.
 \]
 We will restrict our attention to 
\begin{equation}\label{2a3}
a \in [1,3] \ \text{ and } r \in (0, 1/16];
\end{equation}
for most of the construction we do not need to be that restrictive, but notice that
$r=1/16$ corresponds to a dimension $\delta = 1/4$ (because $r^\delta = 1/2$), which
we found out that we cannot reach with this construction, so we will not loose anything here.
Then
 \begin{align}\label{eq:condition-a} 
1- \dfrac{ar}{1-r} \geq \frac45. 
 \end{align}
We will check soon (in Propositions \ref{pro:AR} and \ref{pro:bilip})
that changing the parameters $a_n(\varepsilon)$ in \eqref{2a1} 
yields bi-Lipschitz equivalent sets $K$, and in particular that 
 $K$ is bi-Lipschitz equivalent to the self-similar Cantor set $K^0$ obtained by 
taking all the $a_n(\varepsilon)$ equal to $1$. Then the sets $K$ that we can get with \eqref{2a3} 
are all Ahlfors-regular of dimension $\delta = -\ln(2)/\ln(r) \leq 1/4$.

\medskip 

The verification of Ahlfors regularity for the fractal set $K^0$ is classical and easy; one takes the invariant
measure $\mu$ that gives the same mass $2^{-n}$ to each of the $2^n$ intervals that compose the 
$n$-th approximation of $K^0$,  and checks that $\mu$ satisfies \eqref{1a1}. 
But given our representation \eqref{eq:defK} of $K$, it is slightly easier to check first that 
that $\cE$, endowed with the following ultrametric distance, is an Ahlfors regular metric space, and then prove 
that our parameterization $f$ is bi-Lipschitz. 
For $\varepsilon, \varepsilon' \in \cE$, 
let $m(\varepsilon,\varepsilon')$ be the smallest integer such that $\varepsilon_m\neq \varepsilon'_m$ 
(if $\varepsilon = \varepsilon'$, set $m(\varepsilon,\varepsilon')=\infty$) and define 
$\dist_r(\varepsilon,\varepsilon'):= r^{m(\varepsilon,\varepsilon')-1}$. 
This also defines quotient distances (that we denote the same way) on the finite sets $\cE_n$. 
That is, for $n \geq 1$ and $\ol\varepsilon,\ol\varepsilon'\in \cE_n$, $\ol\varepsilon \neq\ol\varepsilon'$,
we let $m(\ol\varepsilon,\ol\varepsilon')  \in \{1,\dots,n\}$ be the common value of $m(\varepsilon,\varepsilon')$,
where $\varepsilon$ starts with $\ol\varepsilon$ and $\varepsilon'$ starts with $\ol\varepsilon'$, and
$\dist_r(\varepsilon,\varepsilon'):= r^{m(\varepsilon,\varepsilon')-1}$ defines a distance on $\cE_n$.
The following result is classical. 

\begin{pro}\label{pro:AR}
The space $\cE$ endowed with the distance $\dist_r$ is  Ahlfors regular of dimension $\delta :=-\ln 2/\ln r$.
\end{pro}

\begin{proof}
In the space $\cE$, the balls are just $\cE$ and the obvious cells $Q_n(\ol \varepsilon)$, i.e.~the set of all children of $\ol\varepsilon$ with $n \geq 1$ and 
$\ol\varepsilon \in \cE_n$. Let $\sigma$ be the natural probability measure on $\cE$
(the product of the obvious probability measures on $\{ -1, 1 \}$); thus $\sigma(Q_n) = 2^{-n}$,
whereas $Q_n$ is a ball of radius $\rho$ for $\rho \simeq r^n$, so $\sigma(Q_n(\ol \varepsilon)) \simeq \rho^\delta$.
This proves that $\sigma$ is Ahlfors regular of dimension $\delta$, as needed. 
\end{proof}

\begin{pro}\label{pro:bilip}
 Under condition \eqref{eq:condition-a}, for $\varepsilon, \varepsilon' \in \cE$ 
 there holds:
 \begin{equation}\label{eq:bilip-N}
   \Big  (1- \dfrac{ar}{1-r}\Big )r^{m(\varepsilon,\varepsilon')-1} \le |f(\varepsilon)-f(\varepsilon')| \le \dfrac{a r^{m(\varepsilon,\varepsilon')-1}}{1-r},
 \end{equation}
 and the same is true of $f_n$ for $n\ge 0$. In particular, these maps are bi-Lipschitz continuous from $\cE$ 
 (respectively $\cE_{n}$) endowed with the distance $\dist_r$ to their images in $\R\times \{0\} \subset \R^2$.
\end{pro}

\begin{proof}
    Given $\varepsilon \neq \varepsilon'$ in $\cE$, let $m=m(\varepsilon,\varepsilon')$. We can write
    \[
       f(\varepsilon)-f(\varepsilon') = \sum_{n=m}^\infty \Big (\dfrac{a_{n-1}(\varepsilon)}{2} \varepsilon_n  - \dfrac{a_{n-1}(\varepsilon')}{2} \varepsilon'_n \Big) r^{n-1},
    \]
  and as $a_{m-1} (\varepsilon)= a_{m-1}( \varepsilon')$ (since $a_n(\varepsilon)$ depends only on 
  $\varepsilon_1, \dotsc, \varepsilon_n$) and $1\le a_n \le a$ we have
    \[
        \big |f(\varepsilon)-f(\varepsilon') -  \dfrac{a_{m-1}(\varepsilon)r^{m-1}}{2} (\varepsilon_m -\varepsilon'_m)\big| \le \sum_{n=m+1}^\infty a r^{k-1}.
    \]
    Using condition \eqref{eq:condition-a}, we get \eqref{eq:bilip-N}.
\end{proof}

We come to the key technical point of our construction. 
Define $g_0 :=0$ and for $n \in \mathbb{N}\sm\{0\}$ and  $x\in K_n$, define 
\begin{equation}\label{eq:defg}
    g_n(x):= 2^{-n} \sum_{y\in K_n\backslash \{x\}} e(y-x),
\end{equation}
where as previously, $e(x) = \ln |x|$. 
\begin{lem}\label{mainLemma}
    Assume that $r$, and $a$ satisfy condition \eqref{eq:condition-a}, and $r$ is sufficiently small.
  Fix an integer $n\ge 1$ and suppose that the $a_k(\varepsilon)$ have been chosen for $k= 0,\dots, n-2$ 
  and that for all $x \in K_{n-1}$ there holds
  \begin{equation}\label{eq:iter-step-n-1}
    c_{n-1} \le g_{n-1}(x) \le c_{n-1} +  2^{-(n-1)} \dfrac{\ln a}{2},
  \end{equation}
  for some real constant $c_{n-1}$.  Then one can chose the $a_{n-1}(\varepsilon)\in [\,1,a\,]$, 
  depending only on $\varepsilon_1, \dotsc, \varepsilon_{n-1}$, such that there exists $c_n$, 
  with $|c_n-c_{n-1}| \lesssim n 2^{-n}\ln a$ and  satisfying
  \begin{equation}\label{eq:iter-step-n}
    c_n \le g_n (x) \le c_n +  2^{-n} \dfrac{\ln a}{2} 
   \  \ \text{  for $x\in K_n$.}
  \end{equation}
\end{lem}

\begin{proof}
  Fix $x\in K_n$, and denote by $\wh x$ its parent in $K_{n-1}$.
  We can write
  \begin{equation*}
    g_n(x) = g_{n-1}(\wh x) + 2^{-n}\big (\Delta_1 + \Delta_2 + \Delta_3 \big),
  \end{equation*}
  where the terms $\Delta_j$ are defined as follows. 
  The first accounts for the influence of the sibling $\wt x$ of $x$ (the other child of $\wh x$):
  \begin{equation}\label{eq:Delta1-def}
     \Delta_1  =  e(\wt x-x).
  \end{equation}
  The term $\Delta_2$ is the difference of the potential created by the points $y$ in $K_{n-1}$ between 
  the points $\wh x$ and $x$, counted twice because of the $2^{-n+1}$ in the definition of $g_{n-1}$: 
  \begin{equation}\label{eq:Delta2-def}
      \Delta_2 = 2 \sum\limits_{\substack{ y \in K_{n-1}\\ y \neq \wh x}} \big (e( y-x) - e( y-\wh x)\big ).
  \end{equation}
  Finally, $\Delta_3$ collects the difference between the potential at $x$ created by the points $y$ at level $n-1$ and that created by the children of these points $y$:
  \begin{equation}\label{eq:Delta3-def}
    \Delta_3  =  \sum\limits_{\substack{ y \in K_{n-1}\\  y \neq \wh x}} \sum_{z\in \mathrm{Ch}(y)} \big ( e(z-x) - e(y-x)\big ).
  \end{equation}

  Let us first estimate $\Delta_1$; 
  it is the one we can tune without perturbing the rest too much. Write $x= f_n(\varepsilon)$, 
  for some $\varepsilon \in \cE_n$; then  $\wh x = f_{n-1}(\wh\varepsilon)$ (where we remove $\varepsilon_n$ 
  from $\varepsilon$ to get $\wh\varepsilon$) and 
  \begin{equation*}
    \Delta_1 =  \ln (a_{n-1}(\varepsilon)r^{n-1}/2) = \beta_n + \ln (a_{n-1}(\varepsilon)),
  \end{equation*}
  where $\beta_n$ does not depend on $x$. We can choose $a_{n-1}(\varepsilon) \in [\, 1,a\, ]$,
  depending on $\wh \varepsilon$, i.e., the first coordinates $\varepsilon_1, \ldots, \varepsilon_{n-1}$ of 
  $\varepsilon$, so that $\Delta_1$ takes any value between $\beta_n$ and $\beta_n + \ln(a)$.
  We do this so that
  \begin{equation*}
    g_{n-1}(\wh x) + 2^{-n}\Delta_1 = c_{n-1} + 2^{-(n-1)}\dfrac{ \ln a}{2} + 2^{-n}\beta_n.
  \end{equation*}
  So  this choice of $a_{n-1}(\varepsilon)$ compensates exactly the deviation of $g_{n-1}$ from the constant 
  $c_{n-1} +2^{-n} \ln a$ in \eqref{eq:iter-step-n-1}.  Note that $a_{n-1}(\varepsilon)$ depends only on the 
  first $n-1$ terms of $\varepsilon$, so the sibling of $x$ is handled with the same choice of $a_{n-1}(\varepsilon)$. 
  \medskip
  
  There remains to show that $|\Delta_2|+ |\Delta_3|\le \dfrac{\ln(a)}{4}$, 
  where we divide the admissible error by $2$  because we cannot control its sign. 
  To do this, we will use the fact that for $|u|<1$, there holds
  \begin{equation}\label{eq:prop-log}
    |\ln (1+u)| \le \dfrac{ |u|}{1-|u|}. 
  \end{equation}
  Recall that $\wh x = f_{n-1}(\wh\varepsilon)$, where $\wh\varepsilon \in \cE_{n-1}$ 
  records the first $n-1$ letters of the word defining $x$; we decompose $K_{n-1}\backslash \{\wh x\}$ into the annuli
 \[
    \cY_\ell := \big \{ y \in K_{n-1}, y= f_{n-1}(\varepsilon'),\  m(\varepsilon,\varepsilon') = n-\ell \big \},
  \]
  where $\ell$ runs over $\{1,\dots, n-1\}$. It is straightforward to check that $\# \cY_\ell = 2^{\ell-1}$. Fix $y\in \cY_\ell$; 
  by \eqref{eq:bilip-N}, there holds
  \begin{equation}\label{eq:dist-ywhx}
        |y-\wh x| \ge \Big (1- \dfrac{ar}{1-r}\Big )r^{n-\ell-1} 
  \end{equation}
  and also $|x-\wh x|\le a r^{n-1}/2$. So 
  we can write
  \begin{eqnarray*}
    |e(y-x)-e(y-\wh x)| &=& \Big|\ln \dfrac{|y-x|}{|y-\wh x|} \Big | 
                       = \Big|\ln \Big | 1 + \dfrac{x-\wh x}{y-\wh x} \Big| \Big|\\
                        &\le&  \dfrac{a r^{n-1}/2}{\Big (1- \dfrac{ar}{1-r}\Big )r^{n-\ell-1}- a r^{n-1}/2}\\
                        &=&  \dfrac{a r^\ell}{2 \Big (1- \dfrac{ar}{1-r}\Big )- a r^\ell}.
  \end{eqnarray*}
  
  Summing over $\ell$ and $\cY_\ell$ as in $\Delta_2$, we get
  \begin{equation}\label{eq:delta2-bound}
    |\Delta_2| \le 2 \sum_{\ell=1}^{n-1} \dfrac{2^{\ell-1} a r^\ell}{2\Big (1- \dfrac{ar}{1-r}\Big )- a r^\ell}.
  \end{equation}
 Recall from \eqref{eq:condition-a} and \eqref{2a3} that the denominator is at least 
 $\frac85 - a r \geq \frac85 - \frac{3}{16} > 1$; then $|\Delta_2| \le a \sum_{\ell=1}^{+\infty} (2r)^\ell
 < 4r$. Taking $r$ small makes $\Delta_2$ as small as we want (independently of $n$).

  \medskip
  
  In order to estimate $\Delta_3$, we will use the fact that $y\in K_{n-1}$ is the barycenter of its $2$ 
  children $y_1,y_2\in K_n$ so the potential they create is not so different from that created by $y$.  
  Given $y\in K_{n-1}\backslash \{ \wh x\}$, we write 
\begin{align*}
    e(x-y_1) + e(x-y_2) &- 2e(x - \wh y)  
    = \ln \Big ( \dfrac{|x-y_1||x-y_2|}{|x-\wh y|^2}\Big) 
    \\
     &= \ln \Big | \dfrac{(x-\wh y + \wh y - y_1)(x-\wh y + \wh y - y_2)}{(x-\wh y)^2} \Big |\\
                                        &= \ln \Big | 1 + \dfrac{ \wh y - y_1 + \wh y-y_2}{x-\wh y} + \dfrac{(\wh y - y_1)(\wh y-y_2)}{(x-\wh y)^2}\Big |\\
                                        &= \ln  \Big | 1 - \dfrac{(\wh y - y_1)^2}{(x-\wh y)^2}\Big |
  \end{align*}
  because $\wh y - y_2 = - (\wh y - y_1)$. 
  Hence for $\wh y\in \cY_\ell$, \eqref{eq:prop-log} and \eqref{eq:dist-ywhx} yield
     \begin{align*}
       |e(x-y_1) + e(x-y_2) - & 2e(x - \wh y)| =-\ln\Big | 1 - \dfrac{|\wh y-y_1|^2}{|x- \wh y|^2}\Big|\\
                                             &\le \dfrac{|\wh y-y_1|^2}{|x- \wh y|^2-|\wh y-y_1|^2}
                                             \le \dfrac{ a^2r^{2 \ell}}{4\big ( 1 - \dfrac{ar}{1-r}\big )^2-a^2r^{2 \ell}} \, ,
    \end{align*}
 so simplifying and summing over $y\in \cY_\ell$ and $\ell$ gives 
    \begin{equation}\label{eq:delta3-bound}
      |\Delta_3| \le \sum_{\ell=1}^{n-1} \dfrac{2^{\ell-1} a^2 r^{2\ell}}{4\big ( 1 - \dfrac{ar}{1-r}\big )^2-a^2r^{2 \ell}}.
    \end{equation}
    As before, this can be made arbitrarily small (independently of $n$) when $r$ is small. 
    Thus we can choose $r$ so that \eqref{eq:iter-step-n} is satisfied.

    We also said that $|c_n-c_{n-1}| \lesssim n 2^{-n}\ln a$, and this follows from the computations above; the extra
    $n$ comes from the fact that we have a $\ln(r^{n-1})$ hidden in $\Delta_1$.

  \end{proof}
  We just proved that we could construct Cantor sets for which the oscillation of $g_n$ tends to zero when $n$ tends to infinity. 
  We will see in Section~\ref{sec:Green} that these Cantor set satisfy the conclusions of Theorem~\ref{thm:main}.

\subsection{Variants of the construction}
\label{sec:Npieces}

We checked (this is easy but boring) that with $a= 2.217$ and $r= 0.0623$
the estimates above yield $|\Delta_2|+ |\Delta_3| \le \ln a/4$; this means that we can take
these values in the construction (and prove Lemma \ref{mainLemma} and then Theorem~\ref{thm:main}). 
This was obtained with some optimization on $a$ and $r$, using the bounds \eqref{eq:delta2-bound} 
and \eqref{eq:delta3-bound} above.
For those values we get $\delta = -\ln(2)/\ln(r)> 0.2497$, so we can take $\delta_M = 0.2497$
and get examples in Theorem~\ref{thm:main} that are contained in the line.

\ms

We also did some computations with other constructions of $K \subset \R^2$,
which lead to slightly better estimates (because the sets $K$ there are more spread out),
and for some we obtain a dimension $\delta > 0.4$ (for variants of the four corner Cantor set, described below).
However, with these variants, we are not able to get close to $\delta = 1/2$.
Of course this does not mean that other, more subtle constructions, cannot yield $\delta > 1/2$, and this would
not contradict Tolsa's theorem \cite{Tolsa2023ADreg}.

\ms

There are several ways to obtain a Cantor set $K \subset \R^2$ satisfying \eqref{eq:measures-eq}  
with larger dimensions, leaving the line. A simple variant of the construction above consists in replacing
the universal set $\cE = \{ -1, 1 \}^{\bN \sm \{ 0 \}}$ used above with 
$\cE(N) = A_N^{\bN \sm \{ 0 \}}$, where $A_N \subset \R^2 \sim \bC$ is the set of $N$-th roots of
unity. Then we can define $f$ exactly as in \eqref{2a1}, but with coordinates $\varepsilon_k \in A_n$.
This gives a $N$-adic Cantor set in the plane, which as before is bi-Lipschitz equivalent to the fractal set
obtained with $a_{k}(\varepsilon) = 2$, and which is Ahlfors regular of dimension $\delta = - \ln(N)/\ln(r)$.

\ms

For such a set we can keep the statement of Lemma \ref{mainLemma} as it is, with $2$ replaced
with $N$ in all occurrences; the proof goes the same way, except for the following details.
First the precise formula for $\Delta_1$ is slightly different, because now
$\Delta_1 = \sum_{\wt x} e(\wt x-x)$, where the sum runs over the $N-1$ siblings of
$x$, takes a different value. However, the symmetry of $A_N$ implies that this value only depends on our choice of $a_{n-1}(\varepsilon)$, and is thus the same 
for $x$ and its siblings. This allows us to choose $a_{n-1}(\varepsilon)$ as we did before.
For $\Delta_2$ we need to estimate the $e(y-x) - e(y-\wh x)$ slightly differently (we cannot use
the precise algebra with the logs in the same way), but the main terms are still when $y \in K_{n-1}$
is a sibling of $\wh x$, and for the other terms we can use the size of $\nabla e(y-u)$. Finally
$\Delta_3$, which is a term of order $2$ because of the symmetry of $A_n$ (or the fact that
$y$ is the barycenter of its children), is usually somewhat smaller and can be handled with the second derivative.
We skip the computations, and merely record that for $N=4$, $a=2.63$ and $r = 0.033$, we obtained 
the desired estimate $4^n(\Delta_2+\Delta_3) < \frac{3}{32} \ln(a)$, which allows us to take $\delta > 0.406$.
\ms

Taking $N$ very large does not seem to improve the estimates. 
Another tempting variant is to keep $N=2$ as above, but alternatively use the sets $A = \{ -1, 1 \}$
and $A' = i A$ (turned by $90$ degrees). This gives an approximation of $K$ by rectangles whose
sidelengths are in proportion of $\sqrt 2$ (and lay alternatively horizontally and vertically).
This does not seem to be much better than the squares.

 \section[Green's function and the harmonic measure]{Estimating Green's function and the harmonic measure}
 \label{sec:Green}

In this section, we prove the estimates \eqref{eq:measures-eq} and \eqref{eq:green-dist} for the Cantor sets $K$ 
constructed in the previous section. The main tool is the estimate \eqref{eq:iter-step-n} in Lemma \ref{mainLemma}, 
from which we can obtain good estimates on the Green's function and harmonic measure with pole at infinity 
in the domain $\Omega=\R^2\sm K$. 

\ms

Before doing so, we state one useful boundary estimate valid for positive harmonic functions in our domain $\Omega$. This estimate relies on the fact that $\Omega$ is a uniform domain, i.e., it has interior Harnack chains and interior corkscrew points. We use this fact without proof, since it shown in \cite[Lemma 2.2]{DFM-Memoirs} that this is a consequence of the $\delta$-Ahlfors regularity of $\partial \Omega$ with $\delta < 1$, though in the case of our Cantor sets, one could prove the existence of Harnack chains and interior corkscrews rather explicitly.

\begin{thm}[Theorem 1 in \cite{Aikawa01}] \label{thm:bdry_harnack} 
    There exists $A_0 >1$ depending only on $\Omega$ so that for any $x_0 \in \partial \Omega$ and for all $R >0$ sufficiently small, the following holds. Whenever $u, v$ are positive, bounded harmonic functions in $\Omega \cap B(x_0, A_0 R)$ vanishing on $\partial \Omega \cap B(x_0, A_0R)$, then 
    \begin{align*}
        C^{-1} \dfrac{u(y)}{v(y) } \le \dfrac{u(x)}{v(x)} \le C \dfrac{u(y)}{v(y)},
    \end{align*}
    for all $x,y \in \Omega \cap B(x_0, R)$. Here, $C$ depends only on the domain $\Omega$. 
\end{thm}

Another important result which is particularly useful for Cantor sets is the following comparison between Green's function and the value of the harmonic measure in parts of the set which are separated from the rest by an annulus. We found this statement as Lemma 3.4 in \cite{MakarovVolberg1986}.   
\begin{lem}\label{lem:MV}
    Suppose that the set $K\subset \R^2$ is Ahlfors regular, and that there is an annulus $ A:= B(x,(1+\eta) R)\backslash B(x,(1-\eta)R)\subset \R^2\backslash  K$. Fixing $x_0\in \R^2$ with $\dist(x_0,K)\ge 1$, for $y\in A$, there holds:
    \begin{equation}\label{eq:MVlemma}
        \omega^{x_0}(K\cap B(x,R)) \simeq G(x_0, y),
    \end{equation}
    where the constant in the inequalities 
    depends only on $\eta$ and the Ahlfors regularity constant for 
    $K$.
\end{lem}
\begin{proof}
    Fix $z\in \partial B(x,(1+\eta/2)R)$, by Bourgain's Lemma (Lemma 3.4 in \cite{AHMMMTV2016GAFA}) we have $\omega^{z}(K\cap B(x,(1-\eta) R)) \simeq 1$. Fix $z\in \partial B(x,(1+\eta/2)R)$ and $y\in \partial(B(x,(1-\eta/2)R))$, by the maximum principle, there holds $G^z(y)\gtrsim 1$. The Ahlfors regularity condition implies that $K$ satisfies the so-called capacity-density condition 
    (see for instance Section 2.1 of \cite{Tolsa2023ADreg}), 
    thus the logarithmic capacity of $K\cap B(x,(1-\eta)R)$ is comparable to $R$. 
    Let $\wt G$ be Green's function for the domain $\R^2\sm (K\cap B(x,(1-\eta)R))$ (with two variables and where, by invariance 
    under conformal mappings, $\infty$ is a point like any other). The 
    the capacity estimate implies that 
    $\wt G (y,\infty) \lesssim 1$. 
    By the maximum principle and Harnack's inequality, for $z\in \partial B(x,(1+\eta/2)R)$ there holds:
    \[
        G(z,y)\le \wt G(z,y) \simeq \wt G(y, \infty) \lesssim 1.
    \]
    We have proved that $G(z,y)\simeq \omega^{z}(K\cap B(x,R))$ for $y \in \partial B(x,(1-\eta/2)R)$ and $z\in \partial B(x,(1+\eta/2)R)$. Considering both sides of the comparison as bounded nonnegative harmonic functions of $z$ in $\R^2\sm (K\cup B(x,(1+\eta/2)R))$ which vanish on $K\sm B(x,(1+\eta/2)R)$, 
    it follows from Theorem \ref{thm:bdry_harnack}  (and after an inversion) that 
  \eqref{eq:MVlemma} holds for $x_0=z$, not to close to $K$.
\end{proof}

Denote by 
\begin{equation} \label{3b1}
\mu_n = 2^{-n} \sum_{x \in K_n} \delta_x
\end{equation}
the natural probability measure on $K_n$, and then let $\mu$ be the weak limit of the $\mu_n$. 
It is also the pushforward by $f : \cE \to K$ of the natural probability measure $\sigma$ on $\cE$. By Proposition \ref{pro:AR} and as $f$ is bi-Lipschitz, the Ahlfors regularity of $K$ follows.
Notice that $\mu$ is slightly different from $\H^\delta \mres{K}$, but Ahlfors regularity guarantees that the two measures are boundedly equivalent.

\ms 

As a first step of the proof, we show as promised that $\mu$ is the equilibrium measure for $K$. 

	\begin{lem}\label{lg1}
	The function defined by
	\begin{align*} 
	G(y) & = \int \ln(|x-y|) \; d\mu(x)
	\end{align*}
	is continuous on $\R^2$, harmonic on $\R^2 \sm K$, and constant on $K$.
	\end{lem}
    \begin{proof}
     
     First observe that the integral converges beautifully, even when $y \in K$, because if we set
     $A_k(y) = \big\{x \in \R^2 \, ; \, 2^{-k} \leq |y-x| < 2^{-k+1} \big\}$
     for $k \in \bZ$, then the Ahlfors regularity of $\mu$ yields
 \begin{equation} \label{3b2}
\int_{A_k(y)} |\ln(|x-y|)| \; d\mu(x) \leq (|k|+1) \ln(2) \mu(A_k) \leq C (|k|+1) 2^{-k \delta}
\end{equation}
   and the corresponding series converges (because the large $k$ give no contribution).
  Also $G$ is harmonic on $\R^2 \sm K$, by the dominated convergence theorem and because each 
  $x \mapsto \ln(|x-y|)$ is harmonic on $\R^2 \sm \{ y \}$.
  
  Next consider the potential created by the approximating probability measure $\mu_k$, i.e., set
   \begin{equation} \label{3b3}
\tilde{g}_n(y) = \int \ln(|x-y|) \; d\mu_n(x) = 2^{-n} \sum_{x \in K_n} e(y-x)  
        \ \  \text{ for } y\notin K_n.
\end{equation}

      We claim that for any $x_0 \in K_n$ and any $y$ with $|y-x_0| = r^{n-1}/2$
      (we'll see that this is a way to make sure that $y$ is far from $K_n$ and the further
      $K_m$), we have 
    \begin{align}
    c_n + 2^{-n}\ln (r^{n-1}/2) \le \tilde{g}_n(y) \le c_n + 2^{-n}\ln (r^{n-1}/2) + C 2^{-n}, \label{ga1}
    \end{align}
    where $C >1$ is some constant independent of $n$, $x_0$, and $y$. 
    We no longer care about getting precise estimates for $C$. 

    \medskip
    
    Indeed, let $x_0 = f_n(\varepsilon_0) \in K_n$ and such a $y$ be given. 
   Set as previously 
    \begin{align*}
    \cY_\ell & = \{ x = f_n(\varepsilon) \in K_n \setminus \{x_0\} \; : \; m(\varepsilon_0, \varepsilon) = n-\ell\} 
    \end{align*}
    for $\ell=1,\dots , n$. 
    Recall that $m(\varepsilon_0, \varepsilon)$ is the first index at which the sequences 
    $\varepsilon_0$ and $\varepsilon$ disagree, that 
    \begin{align}
    \# \cY_\ell = 2^{\ell-1}, \label{ga2}
    \end{align} 
    and that by \eqref{eq:bilip-N} and \eqref{eq:condition-a},  for  $x \in \cY_\ell$ we have
    \begin{align} \label{ga3}
    |x_0 - x| &= |f_n(\varepsilon_0) - f_n(\varepsilon)|
    \ge \Big ( 1 - \dfrac{ar}{1-r} \Big )r^{m(\varepsilon_0, \varepsilon)-1} 
    \nonumber
   \\ & \ge \Big ( 1 - \dfrac{ar}{1-r} \Big )r^{n-\ell-1} 
   \geq \frac45 \, r^{n-\ell-1} 
    \end{align}
    which is much larger than $|y-x_0| = r^{n-1}/2$ because $\ell \geq 1$. 
    Now 
    \begin{align}
    \tilde{g}_n(y) - g_n(x_0) & = 2^{-n} \left( \ln(|y-x_0|) + \sum_{x \in K_n \setminus \{x_0\}} \ln \left ( \dfrac{|y-x|}{|x_0-x|} \right) \right) \nonumber \\
    & = 2^{-n} \ln(r^{n-1}/2) + 2^{-n}\sum_{x \in K_n \setminus \{x_0\}} \ln  \left |\dfrac{y-x}{x_0-x} \right |, \label{ga4}
    \end{align}
    and we can write
    \begin{align*}
    1 - \dfrac{|y-x_0|}{|x_0-x|} \le \dfrac{\left |y-x \right |}{\left| x_0-x \right| } 
    \le 1 + \dfrac{|y-x_0|}{|x_0-x|}.
    \end{align*}
    This gives for $x \in \cY_\ell$,
    \begin{equation}
    \left | \dfrac{|y-x|}{|x_0-x|} -1 \right|  \le \dfrac{|y-x_0|}{|x_0-x|}  
    \leq \frac58 \, r^{\ell} \leq r
     \label{ga5}
    \end{equation}
    by virtue of \eqref{ga3}. Hence we estimate the second term in \eqref{ga4}:
     \begin{align} \label{ga6} 
     \Big |2^{-n}\sum_{x \in K_n \setminus \{x_0\}} &\ln  \big|\dfrac{y-x}{x_0-x} \big | \Big |
      \le  2^{-n} \sum_{\ell =1}^n \sum_{x \in \cY_\ell } 
     \left| \ln \left( \dfrac{|y-x|}{|x_0-x|} \right) \right | \nonumber \\
    & \lesssim 2^{-n} \sum_{\ell=1}^n \sum_{x \in \cY_\ell} \dfrac{ r^{\ell}}{1-r^{\ell}} 
    \lesssim  2^{-n}r \sum_{\ell=1}^n (2r)^{\ell-1} 
     \lesssim 2^{-n}, 
    \end{align}
     by \eqref{ga2} and since $2r <1$. As the first term in the difference \eqref{ga4} is constant in $y$, the estimate \eqref{ga6} and the conclusion of Lemma \ref{mainLemma} readily give \eqref{ga1}.
    
    \ms
    
    Next, we claim that whenever $ m \ge n$, $x_0 \in K_n$, and $y$ is such that $|y-x_0| = r^{n-1}/2$, there holds 
    \begin{align}
        | \tilde{g}_n(y) - \tilde{g}_m(y) | \le C 2^{-n} \label{ga7}.
    \end{align}
    The proof of \eqref{ga7} is much like the estimate \eqref{ga6}, so we only sketch it. 
    One first checks that the distance estimate \eqref{ga3} still holds when replacing $x\in \cY_\ell$ by any of its descendants; the constant $\frac45$ just becomes a little smaller. The 
    main idea is to write 
    \begin{align*}
        \tilde{g}_m(y) -\tilde{g}_n(y)  & = 2^{-n} \Big(  \sum_{x \in K_n} 2^{-(m-n)} \sum_{z \in \mathrm{Ch}_m(x)} 
        \ln \Big |\dfrac{y-z}{y-x}\Big| \Big),
    \end{align*}
    where $\mathrm{Ch}_m(x)$ are the descendants $z \in K_m$ of $x\in K_n$ (there are $2^{m-n}$ of them). 
     We first take care of the descendants $z$ of $x_0$. Notice that by \eqref{2a1} 
     and \eqref{eq:condition-a}, they satisfy
     \[
    |z-x_0| \leq \sum_{\ell > n} \frac{a}{2} \, r^{\ell -1} \leq a r^n \ll |y-x_0|/2
     \]
  so $\ln \big |\dfrac{y-z}{y-x_0}\big| \leq 2$ and the total contribution to $| \tilde{g}_n(y) - \tilde{g}_m(y)|$
  of all these terms is at most $C 2^{-n}$.
    The rest of the sum can be split into the $\cY_\ell$ for $\ell=1,\dots,n$, and then one notices again 
    that $|z-x|  \ll |y-x|$ to conclude as before. So \eqref{ga7} holds.
    
    \ms
    
    We are ready to prove that $G = e \ast \mu$ is continuous on $\R^2$ and constant on $K$.
    For the continuity at the point $y$, we observe that for $t > 0$ and $z$ in $B(y,t)$, 
    \begin{align*}
    G(y) - G(z) =  \int_{B(y,2t)} &[e(x-y) - e(x-z)] d\mu(x)
    \\ &+ \int_{K \sm B(y,2t)} [e(x-y) - e(x-z)] d\mu(x).
   \end{align*}
    The first term is less than $C \ln(1+t) t^\delta$ by \eqref{3b2}, and (with $t$ fixed) the
    second term tends to $0$ when $z$ tends to $y$, for instance by dominated convergence.
    So $G$ is continuous; a look at the proof the would even show that it is $C^\alpha$ for any 
    $\alpha \in (0,\delta)$ but we do not need this. 
    
    \ms 
    
    We now prove that $G$ is constant on $K$.
     Fix $x_0 = f (\varepsilon_0) \in K$, and let $x_n = f_n(\varepsilon_0) \in K_n$ be the finite 
     approximation of $x_0$.
     Then fix any $y_n$ with $|y_n - x_n| = r^{n-1}/2$. We can write for $m > n$, 
    \begin{align}
    	\left| G(x_0) - c_n  \right| & \le |G(x_0) - G(y_n)| + |G(y_n) - \tilde{g}_m(y_n)| \nonumber \\
    	& + | \tilde{g}_m(y_n) - \tilde{g}_n(y_n) | + | \tilde{g}_n - c_n| \label{ga8}.
    \end{align}
    For $n$ large, the first, third, and fourth terms in \eqref{ga8} can be made arbitrarily small, by the continuity of $G$, \eqref{ga7}, and \eqref{ga1}, respectively. Once $n$ is fixed, the second term in \eqref{ga8} can be made arbitrarily small when $m$ 
  is large, since $\mu$ is the weak limit of the probability measures $\mu_n$ of \eqref{3b1}.
	In view of this estimate, we see simultaneously that $\lim_{n \to \infty} c_n = c_\infty$ exists, and that $G$ is constant and equal to $c_\infty$ on $K$. This completes the proof of Lemma \ref{lg1}.
	\end{proof}
  
 Define $G^\infty$ on $\R^2$ by 
 \begin{align} \label{ga9}
G^\infty(y) = G(y) - c_\infty,
\end{align}   
where $c_\infty$ is the constant value of $G$ on $K$. 
Then $G^\infty$ is continuous on $\R^2$ and harmonic on $\R^2 \sm K$. It is also equivalent
to $\ln(|y|)$ at infinity, and the maximum principle (applied to large balls $B(0,R)$ so that $G > 0$ on
$\d B(0,R)$) show that $G^\infty(y) > 0$ on $\Omega = \R^2 \sm K$. Because of this, $G^\infty$ is equal
(maybe modulo a multiplicative constant) to the Green function on $\Omega$, with pole at infinity.
The uniqueness of the Green function here is a consequence of the Ahlfors regularity and good nontangential
access to $K$, through the H\"older variant of the comparison principle in Theorem \ref{thm:bdry_harnack}.
We refer to \cite{DEM21}, where we know that this is done, by lack of a better reference.
We first prove the estimate \eqref{eq:green-dist} for $G^\infty$.

   \begin{lem}\label{lem:greeninfinity}
   The Green function $G^\infty$ defined by \eqref{ga9} satisfies for $\dist(y, K) \le 1$,
   \begin{align} \label{ga10}
   G^\infty(y) \simeq \dist(y, K)^\delta,
   \end{align}
   where as before, $\delta$ is the Hausdorff dimension of $K$.
   \end{lem}

   \begin{proof}
We give a specific argument, which exploits the fact that $K$ is contained in the line. 
       Fix $R\in (0,1\,]$, and a point $y$ with $\dist(y,K)=R$. 
       Also choose $x_0 \in K$ such that $|x-x_0| = R$. Then consider the corkscrew point
  $y_0 = x_0 \pm (0,R)$. There is a Harnack chain from $y_0$ to $y$, which we can make very simple
  here (because $K \subset \R$, and this is even easier if we choose the sign so that $y_0$ on the same side 
  of $\R$ as $y$), and since the Harnack inequality shows that $G(y)$ and $G(y_0)$ are comparable, it is enough
  to prove \eqref{ga10} for $y_0$. Let us assume, without loss of generality, that $y_0 = x_0 + (0,R)$, and write    
   \[
   G^\infty(y_0) = G(y_0)- c_\infty = G(y_0)-G(x_0) 
            = \int \ln \dfrac{|y_0-z|}{|x_0-z|}d \mu(z).
   \]   
   For the lower bound, observe that $\ln(|y_0-z|/|x_0-z|) \ge 0$ for $z \in K$, because
   $z_0$ lies directly above $x_0$. Then, even with the sole contribution of $B = B(x_0, R/2)$, we get that
   \[
   G^\infty(y_0)  \geq \int_B \ln \dfrac{|y_0-z|}{|x_0-z|}d \mu(z)
   \geq (\ln 2) \mu(B) \gtrsim R^\delta,
   \]  
   because $\mu$ is Ahlfors regular.
   
   \ms 
   
  For the upper bound, we proceed as for the continuity of $G$, and cut $K$ into the annuli
 $ A_k := K \cap (B(x_0, 2^k R) \setminus B(x_0, 2^{k-1} R))$,  $k \in \mathbb{Z}$.
  We start with $k \leq 2$ and say that
  \begin{align*}
\sum_{k \le 2} \int_{A_k} \ln \dfrac{|y_0-z|}{|x_0-z|} \; d\mu(z)  \lesssim \sum_{k \le 2} (|k|+1) \mu(A_k) 
\lesssim \sum_{k \le 2} (|k|+1) (2^k R)^\delta 
\lesssim R^\delta
\end{align*}
by Ahlfors regularity.
For $k > 2$ and $z \in A_k$, it is easy to see that since $y_0 = x_0 + (0, R)$
and $|z-x_0| \geq 2^{k-1} R \geq 2R$, we have
\[1 \leq |y_0-z|/|x_0-z| \leq 1 + C 2^{-k}, 
\]
which yields  $0 \le \ln(|y_0-z|/|x_0-z|) \le 2^{-k}$ and 
\begin{align*}
\sum_{k > 2} \int_{A_k} \ln \dfrac{|y_0-z|}{|x_0-z|} \; d\mu(z) 
&  \lesssim \sum_{k > 2} 2^{-k} \mu(A_k) 
\lesssim R^\delta \sum_{k > 2} 2^{k(1-\delta)} \lesssim R^\delta,
\end{align*}
where again we have used the $\delta$-Ahlfors regularity of $\mu$ and now
the fact that $\delta <1$. We sum all the terms and $G^\infty(y) \lesssim R^\delta$, as in \eqref{ga10}.

\ms

When $K$ is obtained with slightly more complicated constructions, as in Subsection \ref{sec:Npieces},
we need to adapt the proof. 
Only the lower bound looks dangerous, and we need to find some corkscrew points $y_0$ where 
$G(y_0)$ is not too small. One way to do this is to start from a local corner of $K$, and move away in the
opposite direction from $K$. There the fact that $\nabla G(y) = \int_K \frac{y-x}{|y-x|^2} d\mu(x)$, with 
main contributions that do not cancel, helps. But we leave the details.
\end{proof}

 It is a simple consequence of Theorem \ref{thm:bdry_harnack} that \eqref{ga10} also 
 holds for the Green function for $\Omega$ with far enough pole $x \in \Omega$: 
   
  \begin{cor}\label{cor:finite_pole}
  There is a constant $C \geq 1$, which depends only on $\delta$
  (and for the next section, the ambient dimension $n$), such that if
  $x \in \Omega = \R^2 \sm K$, with $\dist(x, K) > C$, then for 
  all $y$ with $\dist(y,K) \le 1$, we have
   \begin{align*}
   C^{-1} \dist(y, K)^\delta \leq G^x(y) \leq C  \dist(y, K)^\delta,
   \end{align*}
   where $G^x(y)$ denotes the Green 
   function for $\Omega$ (and $-\Delta$) with pole at~$x$.
   \end{cor}
    
  \begin{proof}
   Let $A_0$ be the constant from Theorem \ref{thm:bdry_harnack}. 
   Let $z \in \Omega$ be a fixed point with $\dist(z, K) = 1$ and $w  \in K$ with $|z- w| = 1$. 
   Fix $x \in \Omega$ with $|z - x| > A_0(2 + \diam \, K) + 2$. 
   Now for any $y \in \Omega$ with $\dist(y, K) \le 1$,
    the functions $G^\infty$ and $G^x$ are positive and harmonic in $\Omega \cap B(w, A_0(2+\diam \, K) + 1)$ that vanish on $K$. Moreover, $z, y \in B(w, \diam \, K + 2)$ so that applying Theorem \ref{thm:bdry_harnack} implies that 
   \begin{align*}
   \dfrac{G^\infty(y)}{G^x(y)} \simeq \dfrac{G^\infty(z)}{G^x(z)},
   \end{align*}
   which is just a constant (since $z$ is a fixed point). Using \eqref{ga10} we obtain 
   $G^x(y) \simeq  G^\infty(y) \simeq \dist(y, K)^\delta$. The constant depends on the geometric constants in the proof,
   which we can control uniformly as long as $\delta$ stays away from $0$.
   \end{proof}

   \begin{cor}
   There is a constant $C \geq 1$, which depends only on $\delta$
  (and for the next section, the ambient dimension $n$), such that if
  $x \in \Omega = \R^2 \sm K$, with $\dist(x, K) > C$, then 
  \begin{equation} \label{3a14}
C^{-1} \mu \leq \omega^x \leq C \mu,
\end{equation}
where $\omega^x$ denotes the harmonic measure on $K$, with pole at $x$.
  \end{cor}

Notice that $\mu$ is equivalent to $\H^{\delta}\mres K$, so $\omega^x$ is also equivalent 
to $\H^{\delta}\mres K$.

\begin{proof}  
    It suffices to prove 
 that $C^{-1}(B) \mu \leq \omega^x(B) \leq C \mu(B)$ for 
balls $B$ such that $B\cap K= f(Q_n(\ol \varepsilon))$, where $n\ge 1$ and $\varepsilon$ is small enough. 
For these balls, $B\cap K$ is separated from $K\sm B$ by an annulus of bounded modulus, so we can apply 
Lemma \ref{lem:MV} to obtain
    \[
        \omega^x(B) \simeq G^x(y),
    \]
 where $y$ is a point on the boundary of said annulus.
 In particular $\dist(y,K)\simeq R$, where $R$ is the radius of the ball. Thus $G^x(y) \simeq R^\delta$. 
 In turn $\omega^x(B)\simeq R^\delta$, thus, by the properties of Ahlfors regularity, $\omega^x \simeq \mu$.
\end{proof}

\section{Ahlfors regular examples in $\R^n$} 
\label{Sn}

For $n\ge 2$, we construct a set $E= K\times  \mathbb{S}^{n-2}$, where $K$ is a Cantor subset on the line 
$\R\times \{0\}^{n-1}$ defined by as in \eqref{2a1} and $\mathbb{S}^{n-2}$ is the sphere of radius $1$
in the orthogonal plane $\{0\}\times \R^{n-1}$, centered at the origin. 
As in Section~\eqref{sec:construct-line}, we will prove that for $r>0$ small enough, there is a choice of parameters $a(\varepsilon)\in [\, 1,a\,]$ for every $\varepsilon\in \cE$, such that the  set $E$ satisfies the conditions of Theorem~\ref{thm:main2}.

\ms

The steps of the proof are the same as in Section~\ref{sec:construct-line}, the main difference being that at step $m$ in the construction of the set, we replace the discrete sum of Dirac masses by the probability measure 
\[
  \mu_m:= 2^{-m}C(n)^{-1}\sum_{\varepsilon \in \cE_m} \mathscr{H}^{n-2}\mres (\{f(\varepsilon)\}\times \mathbb{S}^{n-2}),
 \]
where $C(n)$ stands for the $(n-2)$ dimensional Hausdorff measure of the unit sphere $\mathbb{S}^{n-2}$.
We need to replace the fundamental solution $\ln|x -y|$ by $|x-y|^{2-n}$. 
We first compute the potential $e(x)$ created at a point $x\in \R^n$ by $\mathscr{H}^{n-2}\mres \mathbb{S}^{n-2}$.
By symmetry, $e(x)$ depends only on the two quantities $t = x_1 \in \R$ and $R = |x'| \in \R^+$ (where $x'$ stands for 
the last $n-1$ coordinates of $x$). And (with a minor abuses of notation)

\begin{equation*}
\begin{aligned} 
  e(t,R) &= C(n)^{-1} \int_{y\in  \mathbb{S}^{n-2}} \dfrac{d \mathscr{H}^{n-2}(y)}{(t^2+ |x'-y|^2)^{(n-2)/2}}
\\ &= C(n)^{-1} \int_{y=(y_1,y_2)\in  \mathbb{S}^{n-2}} \dfrac{d \mathscr{H}^{n-2}(y)}{(t^2+ |R-y_1|^2 +|y_2|^2)^{(n-2)/2}}.
\end{aligned}
\end{equation*}
We shall continue to use the notation $y = (y_1, y_2) \in \mathbb{S}^{n-2}$ where $y_1 \in \R$ and $y_2 \in \R^{n-2}$ 
as in the above. 
Notice that $e(t,R)$ has a singularity (that will be controlled) when $R=1$ and $t=0$.
We start by recording the properties of the function $e(t,R)$ which we will need.
\begin{pro}\label{pro:fundsol} There holds 
  \begin{enumerate}[(i)]
    \item The function $t\mapsto e(t,1)$ is even, continuous and strictly decreasing for $t>0$. 
    Furthermore
    \label{e-continuous}
    \[
        e(t, 1)-e(t',1) \gtrsim 1 \ \text{  for $0<2t<t'$,}
    \] 
    \item $|\nabla e(t,R)| \lesssim (t^2 + (R-1)^2)^{-1/2}$ for $t \in \R$ and $R \geq 0$,
    \label{e-derivative}
    \item $|\partial^2_te(t,1)| \lesssim t^{-2}$ for small $t$.\label{e-secondderivative}
  \end{enumerate}
\end{pro}
\begin{proof}
  Let us start with \eqref{e-continuous}. The only hard part is the increment bound. For $0<2t<t'<1$, we have
  \begin{align*}
    e(t, 1) &-e(t',1) = \int_{y\in \mathbb{S}^{n-2}} \Big (\dfrac{1}{(t^2+ |1-y_1|^2 +|y_2|^2)^{(n-2)/2}} \\
    &\qquad \qquad \qquad \qquad- \dfrac{1}{((t')^2+ |1-y_1|^2 +|y_2|^2)^{(n-2)/2}}  \Big )d \mathscr{H}^{n-2}(y)\\
               &\ge  \int_{\mathbb{S}^{n-2}\cap \{|y-(1,0)| \le t/2\}} \Big (\dfrac{1}{(t^2+ |1-y_1|^2 +|y_2|^2)^{(n-2)/2}} \\
    &\qquad \qquad \qquad \qquad- \dfrac{1}{({(t')}^2+ |1-y_1|^2 +|y_2|^2)^{(n-2)/2}}  \Big )d \mathscr{H}^{n-2}(y)\\
               &\ge \int_{\mathbb{S}^{n-2} \cap \{|y-(1,0)| \le t/2\}} \Big (\dfrac{1}{(t^2 + t^2/4 )^{(n-2)/2}} - \dfrac{1}{ {t'}^{(n-2)}}  \Big ) d \mathscr{H}^{n-2}(y)\\
               &\gtrsim \big( \dfrac{4}{5}\big )^{(n-2)/2} - \big (\dfrac{t}{t'}\big )^{n-2} \gtrsim 1 .
  \end{align*}
  We now turn the bound \eqref{e-derivative}. At the point $p=(t,R,0,\dots,0)$ we estimate, 
  \begin{align*}
    |\nabla e(t,R)|  \lesssim   \int_{y\in \mathbb{S}^{n-2}} \dfrac{ d \mathscr{H}^{n-2}(y)}{ |p-y|^{n-1}}.
  \end{align*}
  Let $r_0:= (t^2 +(R-1)^2)^{1/2}$ be the distance between $p$ and $\{0\} \times \mathbb{S}^{n-2}$ (note that this distance is attained at the point $p_0:=(0,1,0,\dots, 0)$). We decompose $\mathbb{S}^{n-2}$ into annuli $A_k:= \{ y \; : \;  |p-y|\in [\,2^k r_0, 2^{k+1} r_0\,)\}$, where $0 \le k \le -\log_2 r_0 + 10$ (i.e., $k$ is large enough to exhaust the whole sphere). There holds $\mathscr{H}^{n-2}(A_k) \lesssim (2^k r_0)^{n-2}$, and if $y\in A_k$, we have $|p-y | \gtrsim 2^kr_0$, so summing over the annuli, we get
  \begin{align*}
    |\nabla e(t,R)| \lesssim \sum_{k=0}^{\infty} \dfrac{(2^kr_0)^{n-2}}{(2^kr_0) ^{n-1}} \lesssim \dfrac{1}{r_0}.
  \end{align*}
  For the bound \eqref{e-secondderivative} on the second derivative, we proceed similarly.
\end{proof}
Let us now start with the variant of the main lemma.
\subsection{Main Lemma, choice of the parameters}

At a step $m \ge 1$ of construction of our set $E$, we can define the approximate ``Green 
function'' for $x\in K_m\subset \R$ by 
\begin{equation*}
  \wt g _m(x):= 2^{-m} \sum_{y\in K_m\sm \{x\}} e_n(x-y).
\end{equation*}
\begin{lem}\label{mainLemma2}
    Assume that $r$, and $a$ satisfy condition \eqref{eq:condition-a}, and that $r$ is sufficiently small. Then there exists $A >0$ so that the following holds.
  Fix an integer $m\ge 1$ and suppose that the $a_k(\varepsilon)$ have been chosen for $k= 0,\dots, n-2$ 
  and that for all $x \in K_{m-1}$,
  \begin{equation}\label{eq:2iter-step-n-1}
    c_{m-1} \le \wt g_{m-1}(x) \le c_{m-1} +  2^{-(m-1)} A,
  \end{equation}
  for some real constant $c_{m-1}$.  Then one can chose the $a_{m-1}(\varepsilon)\in (1,a)$,
  depending only on $\varepsilon_1, \dotsc, \varepsilon_{m-1}$,
  such that there exists $c_m$
  satisfying
  \begin{equation}\label{eq:2iter-step-n}
    c_m \le \wt g_m (x) \le c_m +  2^{-m} A
   \  \ \text{  for $x\in K_m$.}
  \end{equation}
\end{lem}

\begin{proof}
Here we shall not try to optimize, so the reader may take $a = 2$ if they want. 
  As in the $2$-dimensional case, we assume \eqref{eq:2iter-step-n-1} and 
  cut the next $\wt g_m$ into pieces. 
  If $x\in K_m$ has parent $\wh x\in K_{m-1}$, we can write:
  \begin{align*}
    \wt g_{m}(x) =\wt g_{m-1}(\wh x)  + 2^{-m} (\Delta_1+ \Delta_2+\Delta_3),
  \end{align*}
  where, letting $\wt x$ be the sibling of $x$,
  \begin{equation*}
    \Delta_1 = e(x- \wt x),
  \end{equation*}
  \begin{equation*}
    \Delta_2 = 2 \sum_{y\in K_{m-1}\backslash \{\wh x\}} \big ( e(y-x)-e(y-\wh x)\big),
  \end{equation*}
  and
    \begin{equation*}
    \Delta_3 = \sum_{y\in K_{m-1}\backslash \{\wh x\}}\sum_{z\in \mathrm{Ch}(y)} \big ( e(z-x)-e(y-x)\big).
  \end{equation*}

  There are two things to check: we will start by proving  
  that the right choice of $a(\varepsilon)$ for $\wh x=f(\varepsilon)$ yields
  \begin{equation}\label{eq:2delta1}
    \wt g_{m-1}(\wh x)  + 2^{-m} \Delta_1 = c_{m-1} + 2^{-(m-1)}A + 2^{-m}\beta_m ,
  \end{equation}
  with $\beta_m$ independent of $x$. This will tell us what $A$ can be. We will then show that for $r$ small enough:
  \begin{equation}\label{eq:2delta23}
    |\Delta_2|+|\Delta_3|\leq A/2.
  \end{equation}
  To prove \eqref{eq:2delta1}, we use Proposition \ref{pro:fundsol} part \eqref{e-continuous}, which implies that varying $|x-\wt x|$ between $r^n$ and $a r^n$, $\Delta_1$ can take any value between $e(r^n)$ and $e(a r^n)$. 
  As $0 < 2 r^n \leq a r^n$, 
  there is a constant $A>0$ such that $e(r^n)-e(ar^n)\ge 2A$.
  To prove \eqref{eq:2delta23}, we proceed as in Section~\ref{sec:construct-line},
  replacing the formulae involving logarithms with first and second derivatives of $t\mapsto e(t)$. More precisely, for $\Delta_2$, we use the fact that
  \begin{equation*}
    |e(y-x) - e(y-\wh x)| \lesssim |\partial_t e(y-\wh x)| |x-\wh x| \lesssim \dfrac{ |x-\wh x|}{|y-\wh x|}
  \end{equation*}
  and sum over annuli as above, which yields an arbitrarily small term when $r$ goes to $0$. On the other hand, for $\Delta_3$, we use the estimate
  \begin{equation*}
    |e(y_1-x) + e(y_2 - x) -2 e(\wh y-x)| \lesssim |\partial_t^2e(\wh y-x)| |y_1-\wh y|^2 \lesssim \dfrac{ |y_1-\wh y|^2}{|\wh y-x|^2},
  \end{equation*}
  and this can be handled in the same way.

  Note that since we have \eqref{eq:2delta1} with a fixed $A$, and we have bounds for $\Delta_2$ and $\Delta_3$,
  we also get that $|c_m-c_{m-1}| \leq C 2^{-n} (e(r^m,1))$, as in Lemma \ref{mainLemma}. Indeed it is easy to see deduce 
  from \ref{e-derivative} in Proposition \ref{pro:fundsol} that $e(r^m,1) \leq C (m+1)$, so we can estimate the main piece of
  $\Delta_1$, i.e., $e(r^n)-e(ar^n)$, as in Lemma \ref{mainLemma}.
\end{proof}

\subsection{Control of Green's function}

To conclude, the proof goes as in Section \ref{sec:Green}, using Taylor estimates with the gradient of $e(t,R)$ instead of differences of logarithms. The only real change is that there is no immediate equivalent of Lemma~\ref{lem:MV} to compare Green's function and the harmonic measure. 
However, we can use the very general Lemma 15.28 in \cite{DFM-Mixed}, which provides us with just the right estimate at corkscrew points and with poles far enough away. By the way, the one sided accessibility of $K$ follows from \cite[Lemma 2.2]{DFM-Memoirs}, but it is not harder to prove by hand than in the planar case.

Concerning the lower bound in Lemma \ref{lem:greeninfinity}, our life is easier because by Harnack it is enough to
prove it for corkscrew points, and we can choose the corkscrew points of the form $(t,x')$ with $x' \in \mathbb{S}^{n-2}$.
For those points we can use the symmetry of $G$ and the decreasing property of $e(t,R)$, and proceed as before. 

This completes our description of the proof of Theorem \ref{thm:main2}.

\end{document}